\documentclass[12pt]{amsart}
 \usepackage{amsmath,amssymb,amsbsy,amsfonts,amsthm,latexsym,mathabx,
            amsopn,amstext,amsxtra,euscript,amscd,stmaryrd,mathrsfs,
            cite,array,mathtools,color,bm}
\usepackage{url}
\usepackage[colorlinks,linkcolor=blue,anchorcolor=blue,citecolor=blue]{hyperref}
\usepackage{color}
\usepackage{float}

\hypersetup{breaklinks=true}

\begin{document}

\newtheorem{theorem}{Theorem}
\newtheorem{lemma}[theorem]{Lemma}
\newtheorem{claim}[theorem]{Claim}
\newtheorem{cor}[theorem]{Corollary}
\newtheorem{prop}[theorem]{Proposition}
\newtheorem{definition}{Definition}
\newtheorem{question}[theorem]{Question}
\newcommand{\hh}{{{\mathrm h}}}
\newcommand{\dif}{{{\mathrm d}}}

\numberwithin{equation}{section}
\numberwithin{theorem}{section}
\numberwithin{table}{section}

\def\sssum{\mathop{\sum\!\sum\!\sum}}
\def\ssum{\mathop{\sum\ldots \sum}}
\def\iint{\mathop{\int\ldots \int}}

\def\squareforqed{\hbox{\rlap{$\sqcap$}$\sqcup$}}
\def\qed{\ifmmode\squareforqed\else{\unskip\nobreak\hfil
\penalty50\hskip1em\null\nobreak\hfil\squareforqed
\parfillskip=0pt\finalhyphendemerits=0\endgraf}\fi}

\newfont{\teneufm}{eufm10}
\newfont{\seveneufm}{eufm7}
\newfont{\fiveeufm}{eufm5}
%
%
%
%
\def\frak#1{{\fam\eufmfam\relax#1}}

\newcommand{\bflambda}{{\boldsymbol{\lambda}}}
\newcommand{\bfmu}{{\boldsymbol{\mu}}}
\newcommand{\bfxi}{{\boldsymbol{\xi}}}
\newcommand{\bfrho}{{\boldsymbol{\rho}}}

\def\fK{\mathfrak K}
\def\fT{\mathfrak{T}}

\def\fA{{\mathfrak A}}
\def\fB{{\mathfrak B}}
\def\fC{{\mathfrak C}}

\def \balpha{\bm{\alpha}}
\def \bbeta{\bm{\beta}}
\def \bgamma{\bm{\gamma}}
\def \blambda{\bm{\lambda}}
\def \bchi{\bm{\chi}}
\def \bphi{\bm{\varphi}}
\def \bpsi{\bm{\psi}}

\def\eqref#1{(\ref{#1})}

\def\vec#1{\mathbf{#1}}


\def\cA{{\mathcal A}}
\def\cB{{\mathcal B}}
\def\cC{{\mathcal C}}
\def\cD{{\mathcal D}}
\def\cE{{\mathcal E}}
\def\cF{{\mathcal F}}
\def\cG{{\mathcal G}}
\def\cH{{\mathcal H}}
\def\cI{{\mathcal I}}
\def\cJ{{\mathcal J}}
\def\cK{{\mathcal K}}
\def\cL{{\mathcal L}}
\def\cM{{\mathcal M}}
\def\cN{{\mathcal N}}
\def\cO{{\mathcal O}}
\def\cP{{\mathcal P}}
\def\cQ{{\mathcal Q}}
\def\cR{{\mathcal R}}
\def\cS{{\mathcal S}}
\def\cT{{\mathcal T}}
\def\cU{{\mathcal U}}
\def\cV{{\mathcal V}}
\def\cW{{\mathcal W}}
\def\cX{{\mathcal X}}
\def\cY{{\mathcal Y}}
\def\cZ{{\mathcal Z}}
\newcommand{\rmod}[1]{\: \mbox{mod} \: #1}

\newcommand{\sL}{\ensuremath{\mathscr{L}}}

\def\vr{\mathbf r}

\def\e{{\mathbf{\,e}}}
\def\ep{{\mathbf{\,e}}_p}
\def\em{{\mathbf{\,e}}_m}

\def\Tr{{\mathrm{Tr}}}
\def\Nm{{\mathrm{Nm}}}

 \def\SS{{\mathbf{S}}}

\def\lcm{{\mathrm{lcm}}}

\def\({\left(}
\def\){\right)}
\def\fl#1{\left\lfloor#1\right\rfloor}
\def\rf#1{\left\lceil#1\right\rceil}

\def\mand{\qquad \mbox{and} \qquad}




\hyphenation{re-pub-lished}

\mathsurround=1pt

\def\bfdefault{b}
\overfullrule=5pt

\def \F{{\mathbb F}}
\def \K{{\mathbb K}}
\def \Z{{\mathbb Z}}
\def \Q{{\mathbb Q}}
\def \R{{\mathbb R}}
\def \C{{\mathbb C}}
\def \N{{\mathbb N}}
\def\Fp{\F_p}
\def \fp{\Fp^*}

\def \Prob{{\mathrm {}}}
\def\e{\mathbf{e}}
\def\ep{{\mathbf{\,e}}_p}
\def\epp{{\mathbf{\,e}}_{p^2}}
\def\er{{\mathbf{\,e}}_r}
\def\eM{{\mathbf{\,e}}_M}
\def\eps{\varepsilon}
\def\ord{\mathrm{ord}\,}
\def\vec#1{\mathbf{#1}}

\def \li {\mathrm {li}\,}

\def\mand{\qquad\mbox{and}\qquad}

\newcommand{\commI}[1]{\marginpar{%
\begin{color}{magenta}
\vskip-\baselineskip 
\raggedright\footnotesize
\itshape\hrule \smallskip I: #1\par\smallskip\hrule\end{color}}}
\def\xxx{\vskip5pt\hrule\vskip5pt}

\newcommand{\commL}[1]{\marginpar{%
\begin{color}{magenta}
\vskip-\baselineskip 
\raggedright\footnotesize
\itshape\hrule \smallskip L: #1\par\smallskip\hrule\end{color}}}
\def\xxx{\vskip5pt\hrule\vskip5pt}

\title{Elliptic Curves in Isogeny Classes}

\author[I. E. Shparlinski]{Igor E. Shparlinski} 
\address{Department of Pure Mathematics, University of New South Wales, 
Sydney, NSW 2052, Australia}
\email{igor.shparlinski@unsw.edu.au}

\author[L. Zhao]{Liangyi Zhao} 
\address{Department of Pure Mathematics, University of New South Wales, 
Sydney, NSW 2052, Australia}
\email{l.zhao@unsw.edu.au}

\date{\today}

\begin{abstract} We show that the  distribution of elliptic curves in isogeny classes
of curves with a given value of the Frobenius trace $t$
becomes close to uniform even when $t$ is averaged over very short intervals
inside the Hasse-Weil interval.
\end{abstract}



\maketitle

\section{Introduction}

\subsection{Motivation} 
Let $p> 3$ be prime and let $E$ be an elliptic curve over  the field $\F_p$ of $p$ elements 
given by an affine {\it Weierstrass equation\/} of the form
\begin{equation}
\label{eq:Weier}
y^2 = x^3 + ax+b,
\end{equation}
with coefficients $a,b\in \F_p$, such that $4a^3 + 27 b^2 \ne 0$.

Clearly, there are $p^2 + O(p)$ 
suitable equations of the form~\eqref{eq:Weier}. 
Furthermore, 
they generate $J = 2p  + O(1)$ distinct (that is, non-isomorphic 
over $\F_p$) curves,
and for most of the curves there are exactly $(p-1)/2$ 
distinct equations~\eqref{eq:Weier}, see~\cite{Len}
for a discussion of these properties. 

We recall that the cardinality of the set $E(\F_p)$  of $\F_p$-rational points on any
elliptic curve $E$ satisfies the {\it Hasse--Weil\/} bound which we formulate as 
\begin{equation}
\label{eq:H-W} 
\#E(\F_p)  - p -1 \in [-T, T], 
\end{equation}
where 
$$
T = \fl{2p^{1/2}}, 
$$
we refer the reader to~\cite{Silv} for these and other general properties of
elliptic curves. 

Curves with the same value of $\#E(\F_p)$ 
are said to be isogenous.  Hence, we see from~\eqref{eq:H-W} that there 
are $4p^{1/2} + O(1)$ isogeny classes.

Accordingly, we denote by $I(t)$ the number of distinct isomorphism 
classes in the isogeny class of curves with $\#E(\F_p)=p+1 - t$, for $t \in [-T, T]$.  
Clearly, the average value of $I(t)$ is   
\begin{equation}
\begin{split}
\label{eq:Aver It}
\frac{1}{2T+1} \sum_{-T \le t \le T} I(t) & = \frac{J}{2T+1} \\
& =  \frac{2p + O(1)}{4 p^{1/2} + O(1)}  = 0.5 p^{1/2} + O(1).
\end{split}
\end{equation}
Thus, on average, each isogeny class contains $ 0.5 p^{1/2}$ 
isomorphic curves, which motivates the study of the distribution of the 
values 
$$
\iota(t) = \frac{I(t)}{0.5 p^{1/2}}
$$

Lenstra~\cite{Len} has obtained upper
and lower bounds for this number which show that typically the values of $I(t)$
are of the same order of magnitude as their average value.
In particular,  by~\cite[Proposition~1.9~(a)]{Len}, for any $t \in [-2p^{1/2}, 2p^{1/2}]$
we have
\begin{equation}
\label{eq:It upper}
\iota(t) = O\(\log p \(\log\log p\)^{2}\),
\end{equation}
with an absolute implied constant. We see that the upper bound~\eqref{eq:It upper}
contains some extra logarithmic factors compared to the average  value  $1$, see~\eqref{eq:Aver It}

On the other hand, the result of Birch~\cite{Birch} on the 
{\it Sato-Tate conjecture\/} over the family of all isomorphism classes 
of elliptic curves over $\F_p$ implies the asymptotic formula
(which is uniform over real $\alpha, \beta \in [-1,1]$)
\begin{equation}
\label{eq:It ST}
\sum_{\alpha T \le t \le \beta T} \iota(t) = \mu(\alpha, \beta) T+ o(p), 
\end{equation}
where $\mu(\alpha, \beta)$ is the {\it Sato-Tate density\/} given by
$$
\mu(\alpha, \beta) = \frac{2}{\pi} \int\limits_{\arccos \alpha}^{\arccos \beta} \sin^2 \vartheta \ \dif\vartheta.
$$

Here we obtain an upper bound which in some sense interpolates between~\eqref{eq:It upper}
and~\eqref{eq:It ST} and improves~\eqref{eq:It upper} on average over $t$ over rather short interval.
 
\subsection{Notation}

Throughout the paper,  $p$  always denotes a  prime 
number and $q$ always denotes a prime power, thus we we 
$\F_q$ to denote the field of $q$ elements. 

Furthermore, the letters  $k$, $m$ and $n$ (in both the upper and
lower cases) denote positive integers.

We recall that the notations $U = O(V)$  and $V \ll U$ are all
equivalent to the assertion that the inequality $|U|\le cV$ holds for some
constant $c>0$.
The implied constants in the symbols `$O$' and `$\ll$'  
are absolute. 

For a complex $z$, we  define $\e(2 \pi i z)$ and $\e_r(z) = \e(z/r)$. 

Finally, the notation $z\sim Z$ means that $z$ satisfies the inequalities
$Z< z\leq 2Z$. 

\subsection{Main Result}

We extend the definition of $I(t)$ and $\iota(t)$ to arbitrary 
finite fields of $q$ elements. 

\begin{theorem}
\label{thm:Aver I} Suppose that $R$ is an integer with 
$$
0  < R < 2R < 2\sqrt{q}.
$$ 
We have, 
$$
 \frac{1}{R} \sum_{t \sim R} \iota(t) \ll \frac{\log q}{\sqrt{\log R}} (\log  \log q)^{7/2}.
$$
\end{theorem}

In particular, we see from  Theorem~\ref{thm:Aver I} that for any fixed $\varepsilon$ 
there is a constant $c(\varepsilon)$ such that for $R \ge q^\varepsilon$ we have 
$$
 \frac{1}{R} \sum_{t \sim R} \iota(t) \le c(\varepsilon) (\log q)^{1/2}  (\log \log q)^{7/2}. 
$$
It is also easy to see that it improves on the ``individual'' bound~\eqref{eq:It upper}
starting with $R \ge (\log \log q)^{3 + \varepsilon}$.

\section{Character sums and the large sieve}

\subsection{Basics on exponential and character sums}
\label{sec:basic}

We recall, that for any integers $z$ and  $r \ge 1$,  we have 
the orthogonality relation
\begin{equation}
\label{eq:Orth}
\sum_{-r/2 \le b < r/2} \er(bz) = \left\{\begin{array}{ll}
r,&\quad\text{if $z\equiv 0 \pmod r$,}\\
0,&\quad\text{if $z\not\equiv 0 \pmod r$,}
\end{array}
\right.
\end{equation}
see~\cite[Section~3.1]{IwKow}. 

Additionally, we need the bound
\begin{equation}
\label{eq:Incompl}
\sum_{n=K+1}^{K+L} \er(bn) \ll  \min\left\{L, \frac{r}{|b|}\right\},
\end{equation}
which holds for any integers  $b$, $K$ and $L\ge 1$ with $0 < |b| \le r/2$,
see~\cite[Bound~(8.6)]{IwKow}.

We also refer to~\cite[Chapter~3]{IwKow} for a background on multiplicative characters.

The link between multiplicative characters and exponential sums is given 
by the following well-known identity (see~\cite[Equation~(3.12)]{IwKow}) 
involving Gauss sums
$$
\tau_r= \sum_{v=1}^r \chi_r(v) \er(v)
$$
with the  quadratic character $\chi_r$ modulo an integer $r\ge 1$, that is, a fully multiplicative 
function of both argument which coincides with the Legendre symbol is $r$ is an odd primes 
and also defined for $r=2$ as 
$$
\chi_2(v) 
= \left\{  \begin{array}{rl}
0 ,& \quad \text{if}\  v   \equiv 0 \pmod 2 ; \\
1 ,& \quad \text{if}\   v\equiv  \pm 1 \pmod 8;\\
-1 ,& \quad \text{if}\ v\equiv  \pm 3 \pmod 8.
\end{array} \right.
$$

It is easy to see that for any  integer $v$ with 
$\gcd(v,r) = 1$,  we have
\begin{equation}
\label{eq:tau chi}
 \(\frac{v}{r}\)   \tau_r  =  \sum_{b=1}^r  \(\frac{b}{r}\)  \er(bv).
\end{equation}

By~\cite[Lemma~3.1]{IwKow} we also have:

\begin{lemma} \label{lem:tau size}
We have
$$
|\tau_r | \leq r^{1/2}.
$$
\end{lemma}
 
\subsection{Large sieve for quadratic moduli}

For $t\in [-T,T]$ we define the function $\Delta(t) = 4q- t^2$.

We make use of the following large sieve type
inequality:

\begin{lemma}
\label{lem:Sieve}
Suppose that $0 < R < 2R < 2\sqrt{q}$.  Let $\alpha_1, \ldots, \alpha_N$ be an arbitrary sequence
of complex numbers and let
$$
Z= \sum_{n= 1}^{ N} |\alpha_n|^2 \mand
\cT(x) = \sum_{n= 1}^{N} \alpha_n \e (nx).
$$
 Then,   we have 
\begin{align*} 
\sum_{t \sim R}  & \sum_{\substack{a =1\\ \gcd(a,\Delta(t)) =1}} ^{\Delta(t)}
\left|\cT(a/\Delta(t))\right|^2\\
& \le 
  \( qR + N + \min \left\{ \sqrt{R}N + \sqrt{q}N^{3/4}, \sqrt{N} q\right \}  \) Z (qN)^{o(1)} .
\end{align*}
\end{lemma}

\begin{proof}  Set
\[ S = \{ \Delta(t)  ~:~ t \sim R \} = \{ \Delta(t)~ : ~ Q_0 \leq \Delta(t) \leq Q_1 \}, \]
where $Q_0=4q-4R^2$ and $Q_1=4q-R^2$.  We start by applying~\cite[Theorem~2]{Ba1}.  Let
\[ S_r  = \{ m \in \N  ~:~ mr \in S \} . \]
We clearly have $S_r \subseteq [Q_0/r, Q_1/r]$ and $S_r$ is empty unless $4q$ is a square modulo $r$.   Moreover,  set
\[ A_r(u,k,l) = \max_{Q_0/r \leq y \leq q/r} \# \{ m \in S_r \cap (y, y+u] ~:~m \equiv l \pmod{k} \} , \]
where $u \geq 0$, $k \in \N$, $l \in \Z$ with $\gcd(k,l)=1$.  
Thus $A_r(u,k,l) =0$ unless the system of congruences 
\[ \frac{t^2 -4q}{r} \equiv l \pmod{k} \]
has a solution in $t$.  Following the arguments 
in~\cite[Section~6]{Ba1} with minor changes, we get that if $k \leq \sqrt{N}$, then
\begin{equation} \label{Abound}
 A_r(u,k,l) \le \left( 1+ \frac{\# S_r/k}{R^2/r} u \right) N^{o(1)} .
 \end{equation}
Now using~\cite[Theorem~2]{Ba1} together with a short computation, we get that 
\begin{equation} \label{firstestba} 
\begin{split}
\sum_{\Delta(t)\in S}  \sum_{\substack{a =1\\ \gcd(a,\Delta(t)) =1}} ^{\Delta(t)}
\left|\cT(a/\Delta(t))\right|^2 \ll \left( N + N^{o(1)} (qR + q \sqrt{N}) \right) Z.
\end{split}
\end{equation}

The remainder of the proof  follows closely to those in~\cite{Ba1} and~\cite{BaZha2}.  
Using~\cite[Lemma~8]{BaZha2}, we have the following.  Suppose that $0 < D \leq 1/2$ and $\{ \beta_l \}$ is a sequence of real numbers such that for all $\alpha \in \R$, there is $\beta_l$ such that
\[ \| \beta_l - \alpha \| \leq D . \]
Then,
\begin{equation} \label{Kls}
 \sum_{\Delta(t) \in S}  \sum_{\substack{a =1\\ \gcd(a,\Delta(t)) =1}}^{\Delta(t)}
\left|\cT(a/\Delta(t))\right|^2 \ll K(D) ( N+D^{-1}) Z,
\end{equation}
where
\begin{align*} 
K(D) = \max_{1 \leq l \leq L} \# \{ a/\Delta(t)~:~\Delta(t) \in S ,\   \gcd(a&, \Delta(t))=1, \\
&  \| a/\Delta(t) - \beta_l \| \leq D \} .
\end{align*}
We choose  $\beta_1, \cdots, \beta_L$  in the same way as in \cite{BaZha2} as rationals
of the form 
\[\beta_l =  \frac{b}{r} + \frac{1}{kr^2}, \qquad l =1, \ldots, L, \]
with some $b,r \in \N$, $r \leq D^{-1/2}$, $\gcd(b,r)=1$, $1 \leq b < r$ and $k \in \Z$ with $\lceil r^{-1} D^{-1/2} \rceil \leq |k| \leq \lceil r^{-2} D^{-1} \rceil$.  Now let
\[ \tau = \frac{1}{\sqrt{D}}, \qquad K = \lceil r^{-1} D^{-1/2} \rceil , \qquad \kappa = \lceil r^{-2} D^{-1} \rceil , \]
and
\begin{align*} 
P(\alpha) = \# \{ a/\Delta(t)~:~\Delta(t) &\in S, \\
& \gcd(a, \Delta(t))=1, \ | a/\Delta(t) - \alpha | \leq D \} . 
\end{align*}
Using~\cite[Lemma~9]{BaZha2} and the discussion that follows therein, we get that
\begin{equation} \label{KPbound}
 K(D) \leq 2 \max_{\substack{r \in \N \\ r \leq \tau}} \max_{\substack{b \in \Z \\ \gcd(b,r)=1}} \max _{K \leq k \leq \kappa} P \left( \frac{b}{r} + \frac{1}{kr^2} \right) .
 \end{equation}

Now set 
$$
\psi(x) = \left( \frac{\sin \pi x}{2x} \right)^2.
$$  
Recall that the Fourier transform of $\psi$ is $\widehat{\psi}(t) = \max \{ 0, 1-|t| \}$ which is compactly supported.  We have
\[ P(\alpha) \leq \sum_{\Delta(t) \in S} \ \sum_{a=-\infty}^{\infty} \psi \left( \frac{a-\alpha \Delta(t)}{4qD} \right) . \]
Following the same computation as that in~\cite[Section~4]{BaZha2}, we get
\[ P (\alpha) \ll qR D + qD \sum_{0 < j < (4qD)^{-1}} \left| \sum_{R \leq t \leq 2R} \e ( \alpha j t^2) \right| , \]
and then arrive at
\begin{equation}
\begin{split}  \label{Palphaest1}
& P \( \frac{b}{r} + \frac{1}{kr^2} \)\\
& \quad  \le q R D + \left( \frac{R}{\sqrt{r}} + \sqrt{R} + \sqrt{qDr} + \sqrt{RqD} \right)  \left( \frac{R}{qD} \right)^{o(1)}.
\end{split}
\end{equation}
Next, we now prove an  analogue of~\cite[Equation~(5.1)]{BaZha2}.  Using~\cite[ Lemma~4]{Ba1} and mindful of the bound~\eqref{Abound}, we get
\begin{equation} \label{Palphaest2}
P \left( \frac{b}{r} + \frac{1}{kr^2} \right) \le  \left( 1 + q D r + q R D \right) \left( \frac{1}{D} \right)^{o(1)},
\end{equation}
if $D/2 \leq 1/(kr^2) \leq D$.  Now we  assume that $1/(kr^2) \geq D$.  
Very much like in~\cite[Equation~(5.2)]{BaZha2}, we get
\begin{align*} 
P\left( \frac{b}{r} + z \right) & \ll 1 + \delta^{-1} \int\limits_{Q_0}^{Q_1} \Omega(\delta, y) \dif y \\
& \ll 1 + \delta^{-1} \int\limits_{-\infty}^{4q} \exp \left( - \frac{4q-y}{R^2} \right)  \Omega(\delta, y) \dif y,
\end{align*} 
where 
\begin{align*} 
 \Omega(\delta, y) & = \sum_{y-\delta \leq \Delta(t) \leq y+ \delta} \sum_{\substack{ m \in J(\delta, y) \\ m \equiv - b \Delta (t) \bmod{r} \\ m \neq 0}} 1 \\
 & \ll \sum_{t=-\infty}^{\infty} \psi \left( \frac{t-\sqrt{4q-y}}{2c\delta/R} \right) \sum_{\substack{ m \neq 0 \\ m \equiv -b \Delta(t) \bmod{r}}} \psi \left( \frac{m-yrz}{8\delta rz} \right) ,
\end{align*} 
for some absolute constant $c$ and $J(\delta ,y) = [(y-4\delta)rz, (y+4\delta) rz]$.  Now applying Poisson summation twice and making a change of variables, we get
\begin{equation} \label{3rdP}
\begin{split}
P & \left( \frac{b}{r} + \frac{1}{kr^2} \right)  \ll \frac{\delta z}{R}  \biggl| \sum_{j=-\infty}^{\infty} \frac{\widehat{\psi}(8j\delta z)}{r^*}  \e \left( -4qj \left( \frac{b}{r} + z \right) \right) \\
& \qquad\qquad\qquad \qquad \quad \sum_{l=-\infty}^{\infty} \widehat{\psi} \left( \frac{2cl\delta}{r^* R} \right) G(-j^*b, l; r^*) E(j,l) \biggr| , 
 \end{split}
 \end{equation}
where 
\[ G(a,l;c) = \sum_{d=1}^{c} \e \left( \frac{ad^2+ld}{c} \right) 
\]
denotes the  quadratic Gauss sum with
 \[ r^* = \frac{r}{(j,r)} , \qquad  j^* = \frac{j}{(j,r)}, \]
and 
\[  E(j,l) = \int\limits_0^{\infty} \exp \left( - \frac{u}{R^2} \right) \e \left( juz - \frac{l\sqrt{u}}{r^*} \right) \dif u .  \]
Note that the upper bound in~\eqref{3rdP} is essentially the same as the right-hand side of~\cite[Equation~(5.3)]{BaZha2}.  The computations of~\cite[ Sections~7, 8 and~9]{BaZha2} go through with minor changes and we arrive at the bound
\begin{equation} \label{Palphaest3}
 P(\alpha) \ll \left( R^{1/2} + R D^{1/4} + R^3 D + \sqrt{r} + \frac{RD}{\sqrt{r}z} \right)  \left( \frac{R}{D} \right)^{o(1)} .
 \end{equation}
Now combining~\eqref{Palphaest3} and~\eqref{Palphaest2} in the same manner as in~\cite[Sections~10]{BaZha2}, we derive that
\begin{equation}
\begin{split}
\label{2ndPest}
P(\alpha) &\le  \left( \sqrt{R} + RD^{1/4} +qRD + \sqrt{r} + \sqrt{qRD} r^{1/4} + qDr \right) \\
& \qquad\qquad\qquad\qquad\qquad\qquad\qquad\qquad \qquad\qquad  \left( \frac{q}{D} \right)^{o(1)} .
\end{split}
\end{equation}
Now recall that $r \leq D^{-1/2}$.  We use~\eqref{Palphaest1} if $r > R$ and~\eqref{2ndPest} if $r \leq R$, we get
\[ P (\alpha) \le \left( qRD + \sqrt{R} +\sqrt{q} D^{1/4} + \sqrt{q}R^{3/4} D^{1/2} \right)  \left( \frac{q}{D} \right)^{o(1)} . \]
Setting $D=1/N$ and applying the above bound to~\eqref{KPbound} and then to~\eqref{Kls}, we get
\begin{equation*}
\begin{split}
 \sum_{t \sim R} & \sum_{\substack{a =1\\ \gcd(a,\Delta(t)) =1}} ^{\Delta(t)}
\left|\cT(a/\Delta(t))\right|^2 \\
& \qquad  \le \left( qR + \sqrt{R}N + \sqrt{q} N^{3/4} + \sqrt{q}R^{3/4} N^{1/2} \right) Z (qN)^{o(1)} .
\end{split}
\end{equation*}
Now if $N \leq q\sqrt{R}$, then 
\[  \sqrt{q}R^{3/4} N^{1/2} \leq qR . \]
If $N > q\sqrt{R}$, then
\[ \sqrt{q}R^{3/4} N^{1/2} \leq \sqrt{R}N . \]
Hence
\begin{equation} \label{ls1}
\begin{split}
 \sum_{t \sim R} & \sum_{\substack{a =1\\ \gcd(a,\Delta(t)) =1}} ^{\Delta(t)}
\left|\cT(a/\Delta(t))\right|^2 \\
& \qquad  \le\left( qR + \sqrt{R}N + \sqrt{q} N^{3/4} \right) Z (qN)^{o(1)} .
\end{split}
\end{equation}
The desired result follows by comparing the above to~\eqref{firstestba}. 
\end{proof}

\subsection{Multiple divisor function}

For positive integers $k$, $M$ and $\nu$ we denote by $d_\nu(k,M)$ the 
number of integer factorisations of $k$ of the form
$$
k = m_1\ldots m_\nu, \qquad \mbox{with} \; 1 \le m_1,\ldots, m_\nu \le M. 
$$

We need the following bound, uniform with respect to all parameters
which is due to Garaev~\cite[Lemma~8]{Gar}. 

\begin{lemma}
\label{lem:Div} 
We have
$$
\sum_{k \le M^\nu} d_{\nu,M}(k)^2 \le M^\nu \(\frac{e \log M}{\nu} + e\)^{\nu^2}. 
$$
\end{lemma}

\subsection{Bound of character sums on average}

Let $\xi_t$ be the quadratic character modulo $\Delta(t)$ (as 
defined in Section~\ref{sec:basic}). We consider the character sums 
$$
S_t(N) =  \sum_{n=1}^{N} \xi_t(n).
$$ 
We follow the ideas of Garaev~\cite[Theorem~3]{Gar} to obtain the 
following result. Throughout this section, we define
\begin{equation}
\label{eq:X def}
X = \max \left\{ q \sqrt{R}, \frac{q^2}{R^2} \right\} .
\end{equation}

\begin{lemma}
\label{lem:Aver Char Sum} 
Suppose that $R$ is an integer with $0 < R < 2R < 2\sqrt{q}$.  
For positive integer $L$ and $\nu$, such that
\begin{equation} \label{Lcond}
L^{\nu} \ge X
\end{equation}
we have
$$
 \sum_{t \sim R}  \max_{N \sim L}  |S_t(N)| \le   R^{1-1/(4\nu)} L^{1+o(1)}   \(\frac{e \log (2L)}{\nu} + e\)^{\nu/2}   ,
$$
as $L\to \infty$. 
\end{lemma}

\begin{proof} For each integer $t \in [-T,T]$ we choose  $N_t \sim L$, 
with 
$$
|S_t(N_t)| = \max_{N \sim L} |S_t(N)|.
$$
Let $M = 2L$. 
Using~\eqref{eq:Orth}, for $N_t \sim L$ we write
\begin{equation*}
\begin{split}
\sum_{n=1}^{N_t} \xi_t(n) &=
 \sum_{m=1}^M \xi_t(m)\frac{1}{M} \sum_{n=1}^{N_t} \sum_{b=-L}^{L-1}  \eM(b(m-n)) \\
&= \frac{1}{M} \sum_{b=-L}^{L-1}   \sum_{n=1}^{N_t} \eM(-bn)
 \sum_{m=1}^M  \xi_t(m) \eM(bm). 
\end{split}
\end{equation*}
Recalling~\eqref{eq:Incompl}, we derive
$$
\left|\sum_{n=1}^{N_t} \xi_t(n)\right| \ll
 \sum_{b=-L}^{L-1}  \frac{1}{|b|+1}   
\left|\sum_{m=1}^M  \xi_t(m) \eM(bm)\right| .
$$
Therefore, 
\begin{equation}
\label{eq:S Wb}
 \sum_{t \sim R}  \max_{N \sim L} |S_t(N)|   = 
 \sum_{t \sim R}  \left|\sum_{n=1}^{N_t} \xi_t(n)\right| \ll \sum_{b=-L}^{L-1}  \frac{1}{|b|+1}   W_b 
\end{equation}
where 
$$
W_b =   \sum_{t \sim R}   \left|\sum_{m=1}^M  \xi_t(m) \eM(bm)\right| .
$$

By H{\"o}lder's inequality, we get the bound
$$
W_b^{2\nu} \le R^{2\nu   -1}  \sum_{t \sim R} 
\left|\sum_{m=1}^M  \xi_t(m) \eM(bm)\right|^{2\nu}.
$$
We now note that 
$$
\(\sum_{m=1}^M  \xi_t(m) \eM(bm)\)^\nu = \sum_{k=1}^{K} \rho_{b,\nu}(k)  \xi_t(k) , 
$$
 where $K = M^\nu$ and 
$$
\rho_{b,\nu}(k) = \sum_{\substack{m_1,\ldots,m_\nu=1\\ m_1\ldots m_\nu = k}}^M 
\eM(b( m_1+\ldots + m_\nu)). 
$$
Using~\eqref{eq:tau chi}, we  write
$$
\(\sum_{m=1}^M  \xi_t(m) \eM(bm)\)^\nu = \sum_{k=1}^{K} \rho_{b,\nu}(k)  
 \frac{1}{\tau_t}  \sum_{v=1}^{\Delta(t)} \xi_t(v) \e_{\Delta(t)}(kv).
$$
Changing the order of summation, 
by Lemma~\ref{lem:tau size} and the Cauchy inequality, we obtain,
$$ 
\left|\sum_{m=1}^M  \xi_t(m) \eM(bm)\right|^{2\nu} \le
 \sum_{v=1}^{\Delta(t)} \left| \sum_{k=1}^{K} \rho_{b,\nu}(k)  \e_{\Delta(t)}(kv)\right|^2.
$$
Thus,  
$$
W_b^{2\nu} \le R^{2\nu   -1} 
 \sum_{t \sim R}   \sum_{\substack{v=1 \\ \gcd(v,\Delta(t))=1}}^{\Delta(t)} \left| \sum_{k=1}^{K} \rho_{b,\nu}(k)  
\e_{\Delta(t)}(kv)\right|^2 
$$
Using that $|\rho_{b,\nu}(k)| \le d_{\nu,M}(k)$ and recalling  Lemma~\ref{lem:Div}, 
we now derive from Lemma~\ref{lem:Sieve},
estimating $Z$ as
$$
Z \ll M^\nu    \(\frac{e \log M}{\nu} + e\)^{\nu^2},
$$   
that 
\begin{equation*}
\begin{split}
W_b^{2\nu} \le & R^{2\nu   -1} M^\nu  \(\frac{e \log M}{\nu} + e\)^{\nu^2} \\
&\quad 
 \( qR + M^\nu + \min \left\{ \sqrt{R}M^\nu + \sqrt{q}M^{3\nu/4}, M^{\nu/2} q \right\}  \)   \(M^{\nu}q\)^{o(1)}. 
\end{split}
\end{equation*}
Always using the first term in the above minimum, we obtain 
\begin{equation*}
\begin{split}
 W_b^{2\nu} \le R^{2\nu   -1}  & M^\nu  \(\frac{e \log M}{\nu} + e\)^{\nu^2}  \\
 & \( qR + \sqrt{R} M^\nu + \sqrt{q} M^{3\nu/4}    \)   \(M^{\nu}q\)^{o(1)}. 
\end{split}
\end{equation*}
Mindful of the inequality $L^{\nu} \ge \max \{ q\sqrt{R} , q^2/R^2 \} $, we see that in the above inequality, the term $\sqrt{R} M^{\nu}$ dominates over $qR+\sqrt{q} M^{3\nu/4}$ .  So the last display simplifies further as
$$
W_b^{2\nu}\le R^{2\nu   -1/2}  M^{(2+o(1))\nu} \(\frac{e \log M}{\nu} + e\)^{\nu^2}. 
$$
Therefore
\begin{equation*}
\begin{split}
W_b \le  R^{1-1/(4\nu)} L^{1+o(1)}   \(\frac{e \log (2L)}{\nu} + e\)^{\nu/2}    .
\end{split}
\end{equation*}
Inserting this into~\eqref{eq:S Wb}, we conclude the proof.  
\end{proof} 

We remark here that simply using the classical large sieve inequality, the sum in Lemma~\ref{lem:Sieve} is
\[\sum_{t \sim R}   \sum_{\substack{a =1\\ \gcd(a,\Delta(t)) =1}} ^{\Delta(t)}
\left|\cT(a/\Delta(t))\right|^2 \leq (q^2 +N) Z . \]
Using this bound, one can still arrive at a non-trivial estimate in Lemma~\ref{lem:Aver Char Sum}, but the condition in \eqref{Lcond} needs to be replaced by $L^{\nu} > q^2$.  Moreover, as in \cite{Zhao}, we can conjecture that
\[
\sum_{t \sim R}   \sum_{\substack{a =1\\ \gcd(a,\Delta(t)) =1}} ^{\Delta(t)}
\left|\cT(a/\Delta(t))\right|^2  \leq (qR + N) (qN)^{o(1)}Z . \]
Having this bound at one's disposal, a non-trivial bound in Lemma~\ref{lem:Aver Char Sum} can be obtained as soon as $L^{\nu} > qR$.

We now obtain a bound on character sums $S_t(N)$ on average starting from rather short intervals of 
length $N \sim L$ with 
\begin{equation}
\label{eq:L large}
\frac{\log L}{\sqrt{\log \log L}} \geq \frac{4 \log X}{\sqrt{\log R}},  
\end{equation}
where, as before, $X$ is given by~\eqref{eq:X def}.

\begin{cor}
\label{cor:Aver Char Sum 1} Suppose that $R$ is an integer with $0 < R < 2R < 2\sqrt{q}$.  
For all integers $L$ satisfying~\eqref{eq:L large}, 
we have
$$
 \sum_{t \sim R}  \max_{N \sim L}  |S_t(N)| \le  R L \exp\(-(7/8 +o(1))  \sqrt{\log R \log  \log L}\) 
$$
as $L \to \infty$. 
\end{cor}

\begin{proof} We assume that $L$ is sufficiently large and choose 
$$
\nu = \rf{\frac{1}{4} \sqrt{\frac{\log R}{ \log \log L}}} +3.
$$
We see from~\eqref{eq:L large}  that
$$
\nu \log L \ge 4 \nu  \log X \sqrt{ \frac{\log  \log L}{\log R}} \ge  \log X,
$$
and thus Lemma~\ref{lem:Aver Char Sum} is applicable with this parameter $\nu$. 
Furthermore
\begin{equation*}
\begin{split}
 \(\frac{e \log (2L)}{\nu} + e\)^{\nu/2}&  \le 
  \(  \log (2L)   + e\)^{\nu/2}\\
 & = \exp\((1/2+o(1))  \nu  \log  \log (2L) \)\\
 & = \exp\((1/8+o(1)) \sqrt{\log R \log \log L}\). 
\end{split}
\end{equation*}
We also have 
$$
R^{1/ 4\nu} = \exp\(\frac{1}{4\nu }  \log R \)=   \exp\((1 +o(1)) \sqrt{\log R \log \log L}\).
$$
Thus  recalling
Lemma~\ref{lem:Aver Char Sum}, we conclude the proof. 
\end{proof} 


\section{Proof of Theorem~\ref{thm:Aver I}}

\subsection{Isogeny classes and $L$-functions}

Let $\sL(t) = L(1, \chi_t)$ be the value of the $L$-function at $s=1$ 
which correspond to the quadratic character modulo $\Delta_t=4p-t^2$ and
let  $\sL^*(t) = L(1, \chi_{t}^*)$ be the value of the $L$-function at $s=1$ 
which corresponds to the quadratic character modulo $\Delta_{t}^*= \Delta_t/f_t^2$ 
where $f$ is the largest integer such that $f^2 \mid \Delta$ and 
$\Delta_0 \equiv 0,1 \pmod 4$. 

We need the following two results which follow from the identities and estimates 
given by Lenstra~\cite[Pages 654--656]{Len}. We obtain the following 
two relations. 
First we have
\begin{equation}
\label{eq:I and L}
I(t) = \frac{\sqrt{\Delta_t}}{2 \pi} \sL^*(t) \psi(f_t)
\end{equation}
where $\psi(f)$ is an explicitely defined function which for an integer $f \ge 1$ 
satisfies the bound 
\begin{equation}
\label{eq:psi}
\psi(f) \ll \(\log \log (f+2)\)^2.
\end{equation}
Then,  we also have 
\begin{equation}
\label{eq:L and L}
\sL(t)  = \sL^*(t)  \prod_{\substack{\ell \mid f\\\ell~\text{prime}}} \(1 - \frac{\chi_t^*(\ell)}{\ell}\),
\end{equation}
(which follows from the Dirichlet product formula for $L$-functions).

Clearly, 
$$
 \prod_{\substack{\ell \mid f\\\ell~\text{prime}}} \(1 - \frac{\chi_t^*(\ell)}{\ell}\)^{-1}
 \le\prod_{\substack{\ell \mid f\\\ell~\text{prime}}} \(1 - \frac{1}{\ell}\)^{-1}
$$
Using the fact that $f$ has at most $O(\log f)$ prime factors and recalling the
Mertens formula, see~\cite[Equation~(2.16)]{IwKow}, we obtain 
$$
 \prod_{\substack{\ell \mid f\\\ell~\text{prime}}} \(1 - \frac{\chi_t^*(\ell)}{\ell}\)^{-1} 
 \ll \log \log (f+2).
$$
Hence, collecting~\eqref{eq:I and L}, \eqref{eq:psi} and~\eqref{eq:L and L} together we obtain

\begin{lemma}
\label{lem:It Lt} 
We have, 
$$
I(t) \ll \sqrt{\Delta_t} \sL(t) (\log \log q)^3.
$$ 
\end{lemma}

We remark that with a little bit of care, one can be more precise with the 
double logarithmic function in Lemma~\ref{lem:It Lt}. However, to streamline the
exposition, we ignore this minor savings and concentrate on the powers
of $\log q$. 

\subsection{Bounds of $L$-functions} We now give abound on the average value of $\sL(t) $
over dyadic intervals. 

\begin{lemma}
\label{lem:Aver L} Suppose that $R$ is an integer with 
$$
(\log q)^2  < R < 2R < 2\sqrt{q}.
$$  
Then we have
$$
 \sum_{t \sim R} | \sL(t)| \ll  \log q \sqrt{ \frac{\log  \log q}{\log R}}
$$
as $R\to \infty$. 
\end{lemma}

\begin{proof} Let  $X$ be defined by~\eqref{eq:X def}. 
We write $L_j = 2^j$ for $j = J, J+1, \ldots$, where $J$ is defined by the inequalities 
$$
\frac{\log 2^{J-1}}{\sqrt{\log \log 2^{J-1}}} < \frac{4 \log X}{\sqrt{\log R}} \le \frac{\log 2^J}{\sqrt{\log \log 2^J}} .
$$
Thus we can apply Corollary~\ref{cor:Aver Char Sum 1}  
with $L = L_j$ for any $j \ge J$.

We also define $I$ by the conditions 
$$
2^{I-1} \le q^2 < 2^I
$$
Since for any integer $N \ge 1$,  we trivially have
$$
\left| \sum_{n =1}^N \xi_t(n)\right| \le \Delta_t \le q
$$
By partial summation, we immediately obtain 
\begin{equation}
\label{eq:large}
\sum_{n \ge L_I} \frac{\xi_t(n)}{n} \ll 1. 
\end{equation}

For $J \le j \le I$,  invoking Corollary~\ref{cor:Aver Char Sum 1} and using partial summation again, 
we derive
$$ 
\sum_{t \sim R} \left| \sum_{L_j \le  n < L_{j+1}} \frac{\xi_t(n)}{n}\right| \le
 R   \exp\(-(7/8 +o(1))  \sqrt{\log R \log  j}\)  .
$$ 
Hence
\begin{align*} 
\frac{1}{R}\sum_{t \sim R} \left| \sum_{L_J \le n \le L_I} \frac{\xi_t(n)}{n}\right| & \ll 
\sum_{J \le j \le I}  \exp\(-(7/8 +o(1))  \sqrt{\log R \log  j}\) \\
& \le \sum_{2 \le j \le I}  \exp\(-(7/8 +o(1))  \sqrt{\log R \log  j}\)\\
& \le \sum_{h =0}^{H}   \exp\(h -(7/8 +o(1))  \sqrt{h \log R}\), 
\end{align*}
where $H = \rf{\log I}$. 
Since $I \sim  2\log q$ we have $H \sim \log \log q$. 
Thus for $\log R > 2 \log \log q$  and $h \le H$, we see that 
$$
\frac{7}{8}   \sqrt{h \log R} \ge \frac{7\sqrt{2}}{8}  h >  \frac{6h}{5}
$$
and we obtain 
\begin{equation}
\begin{split}
\label{eq:med}
\frac{1}{R}\sum_{t \sim R} \left| \sum_{L_J \le n \le L_I} \frac{\xi_t(n)}{n}\right| &
\le \sum_{h =0}^{H}  \exp\(h -(7/8 +o(1))  \sqrt{h \log R}\)   \\
& \le \sum_{h =0}^{H}  \exp\(-(1/5 +o(1))h \)    \ll 1.
\end{split} 
\end{equation}
Finally we also use the trivial bound
\begin{equation}
\label{eq:small}
 \sum_{n < L_J} \frac{\xi_t(n)}{n} \ll \log L_J
 \ll \log X \sqrt{ \frac{\log  \log X}{\log R}}.
\end{equation}
Combining~\eqref{eq:large}, \eqref{eq:med}, and~\eqref{eq:small}, 
and using that $q \le X \le q^2$, 
we conclude the proof. 
\end{proof}

\subsection{Concluding the proof}

We observe that for  $R \le (\log q)^2$ the result follows immediately 
from~\eqref{eq:It upper} (which can be extended to arbitrary fields without 
any changes in the argument). 

Otherwise we combine Lemmas~\ref{lem:It Lt}  and~\ref{lem:Aver L} 
to derive the desired estimate.

%
%
%


\begin{thebibliography}{99}

%
%
\bibitem{Ba1}  S. Baier, `On the large sieve with sparse sets of moduli', {\it J. Ramanujan Math. Soc.}, {\bf 21} (2006),  279--295.

\bibitem{BaZha2}  S. Baier and L. Zhao, `An improvement for the large sieve for square moduli', {\it   J. Number Theory\/}, {\bf 128} (2008),  154--174.


\bibitem{Birch} B.~J.~Birch, `How the number of points of an elliptic curve
over a fixed prime field varies', {\it J.\ Lond. Math.  Soc.}
\textbf{43} (1968),    57--60.

\bibitem{Gar} M.~Z.~Garaev,  `Character sums in short intervals and
the multiplication table modulo a large prime', {\it Monat.
Math.\/}, {\bf  148} (2006),   127--138.

  \bibitem{IwKow} H. Iwaniec and E. Kowalski,
{\it Analytic Number Theory\/}, Amer.  Math.  Soc.,
Providence, RI, 2004.

\bibitem{Len} H. W. Lenstra, `Factoring integers with elliptic
curves',
{\it Annals of Math.\/}, {\bf 126} (1987), 649--673.  

\bibitem{Silv}  J. H. Silverman, {\it The Arithmetic of Elliptic
Curves\/}, Springer-Verlag, Berlin, 2009. 

\bibitem{Zhao}
L. Zhao, `Large sieve inequality for characters to square moduli', 
{\it Acta Arith.\/}, {\bf  112} (2004), 297--308.

\end{thebibliography}
\end{document}